\documentclass[11pt,a4paper,reqno]{amsart}
\usepackage[hidelinks]{hyperref}
\pdfoutput=1
\usepackage{amssymb,eulervm,palatino,interval,amsaddr,enumitem,fullpage}
\usepackage[USenglish]{babel}

\setenumerate{label=\textup{(\roman*)},itemsep=2pt,topsep=2pt}

\allowdisplaybreaks[4]
\newcommand{\doi}[1]{\url{https://doi.org/#1}}
\newcommand{\isbn}[1]{\url{https://isbnsearch.org/isbn/#1}}
\newcommand{\arxiv}[1]{\href{https://arxiv.org/abs/#1}{preprint arXiv:#1}}
\newcommand{\web}[1]{\url{#1}}

\newtheorem{thm}{Theorem}
\newtheorem{lemma}[thm]{Lemma}
\newtheorem{prop}[thm]{Proposition}
\newtheorem{rem}[thm]{Remark}
\newtheorem{cor}[thm]{Corollary}
\newtheorem{df}[thm]{Definition}

\newcommand{\A}{\mathcal{A}}

\newcommand{\C}{\mathbb{C}}
\newcommand{\Z}{\mathbb{Z}}
\newcommand{\N}{\mathbb{N}}
\newcommand{\R}{\mathbb{R}}
\newcommand{\inner}[1]{\left<#1\right>}

\renewcommand{\hbar}{\hslash}

\begin{document}
\leftmargini=2em

\title{\vspace*{-1cm}Isomorphisms of quantum spheres}

\author[F.~D'Andrea]{\vspace*{-5mm}Francesco D'Andrea}

\address{Dipartimento di Matematica e Applicazioni ``R.~Caccioppoli'' \\ Universit\`a di Napoli Federico II \\
Complesso MSA, Via Cintia, 80126 Napoli, Italy}


\begin{abstract}
For $n\in\N$ and $q\in\rinterval{0}{1}$, the Vaksman-Soibelman quantum sphere $S^{2n+1}_q$ is described by an associative *-algebra $\A(S^{2n+1}_q)$ deforming the algebra of polynomial functions on the $2n+1$ dimensional unit sphere.
Its C*-enveloping algebra is known to be independent of the deformation parameter $q$. In contrast to what happens in the C*-algebraic setting, we show here that, for all $q,q'$ in the above range, $\A(S^{2n+1}_q)$ is isomorphic to $\A(S^{2n+1}_{q'})$ if and only if $q=q'$. From a geometric point of view, this means that quantum spheres corresponding to different values of the deformation parameter are isomorphic as noncommutative topological spaces but not as noncommutative algebraic varieties.
\end{abstract}

\maketitle

\section{Introduction}

The idea of noncommutative geometry is to characterize spaces by their commutative algebras of functions and, inspired by quantum mechanics, to look at noncommutative algebras as dual objects of certain virtual ``noncommutative'' or ``quantum'' spaces \cite{Con94}. This idea has many different manifestations. The two that are relevant to us are the following.
In algebraic geometry, the category of affine schemes is equivalent to the opposite category of commutative rings. In topology, the category of compact Hausdorff spaces is equivalent to the opposite category of commutative unital C*-algebras.

We are interested in (non-commutative) associative unital *-algebras over $\C$,
and our typical examples will be quotients of free algebras by ideals generated by non-commutative polynomials. Motivated by the above considerations, we think of such an algebra as describing a noncommutative algebraic variety, and interpret its C*-enveloping algebra (if it exists) as describing the underlying noncommutative topology.

\smallskip

The prototypical example is given by the Weyl algebra (the canonical commutation relation of quantum mechanics), given by
\[
\A(\R^2_\hbar):=\C\!\inner{p,q}/\inner{pq-qp+\mathrm{i}\hbar 1} ,
\]
with involution defined by $p^*=p$ and $q^*=q$, and with $\hbar\in\R$ the deformation parameter.
It is not difficult to prove that $\A(\R^2_{\hbar})$ and $\A(\R^2_{\hbar\hspace{-1pt}'})$ are isomorphic for all $\hbar,\hbar\hspace{-1pt}'\in\R\smallsetminus\{0\}$ (it is a simple exercise to construct the explicit isomorphism).

\smallskip

The next example is the algebra $\A(\mathbb{T}^2_\theta)$ of the noncommutative torus, where $\theta\in\R$ is the deformation parameter.
This is generated by two unitaries $u$ and $v$ with commutation relation
\begin{equation}\label{eq:UV}
uv-e^{2\pi i\theta}vu \,.
\end{equation}
Since the relation \eqref{eq:UV} only depends on the fractional part of $\theta$, and there is an obvious isomorphism $\A(\mathbb{T}^2_\theta)\cong \A(\mathbb{T}^2_{1-\theta})$ (exchanging $u$ and $v$), without loss of generality we can assume that $\theta\in\interval{0}{\frac{1}{2}}$.
We denote by $C(\mathbb{T}^2_\theta)$ the C*-enveloping algebra.

For $\theta,\theta'\in\interval{0}{\frac{1}{2}}$, it is known that 
$C(\mathbb{T}^2_\theta)$ and $C(\mathbb{T}^2_{\theta'})$ are isomorphic only if $\theta=\theta'$ (this is true both for rational and irrational deformation parameter, see e.g.~\cite[Sect.~12.2]{GBVF01} and references therein).
Since isomorphic *-algebras would have isomorphic C*-enveloping algebras, we deduce that 
the polynomial algebras $\A(\mathbb{T}^2_\theta)$ and $\A(\mathbb{T}^2_{\theta'})$ are also not isomorphic if $\theta\neq\theta'$.

\smallskip

We finally arrived at the class of examples studied in this paper: the odd-dimensional quantum spheres of Vaksman-Soibelman. For $n\in\N$ and $q\in\interval{0}{1}$ a deformation parameter, 
the $2n+1$-dimensional quantum sphere $S^{2n+1}_q$ is described by a filtered quadratic algebra
$\A(S^{2n+1}_q)$ that was introduced in \cite{VS91} as quantum homogeneous spaces of the quantum unitary groups
(cf.~Sect.~\ref{sec:1} for the definition).
For $n=0$, one gets the algebra of polynomial functions on the unit circle,
and for $n=1$ the quantum space underlying Woronowocz's quantum $SU(2)$ group.
The ``classical limit'' corresponds to the value $q=1$ of the parameter.

The C*-enveloping algebra $C(S^{2n+1}_q)$ is independent of $q$, for all $0\leq q<1$. This was originally proved by Sheu in \cite{She97} using groupoid \mbox{C*-algebras}. Later on, Hong and Szyma{\'n}ski in \cite{HS02} gave an independent proof using graph C*-algebras. For a proof using graph C*-algebras one can see also the survey \cite{Dan23}, while the connection between the graph and the groupoid pictures is discussed in \cite{Dan24a}.

Thus, quantum spheres corresponding to different values of the deformation parameter are isomorphic as noncommutative topological spaces.
It is only natural to wonder whether they are also isomorphic as noncommutative algebraic varieties, i.e.~whether
the *-algebra $\A(S^{2n+1}_q)$ and $\A(S^{2n+1}_{q'})$ are isomorphic for $q\neq q'$.
The difficulty of constructing an isomorphism between two such algebras suggests that the answer to this question is negative. Surprisingly, there is no proof of this result in the literature (there is a proof in \cite{Kra08} that, for a fixed $q$, the algebra $\A(S^2_{q,s})$ of a Podle\'s sphere does depend on the additional parameter $s$, but this proof cannot be directly applied to our case).
The aim of this paper is to fill this gap.

\medskip

Let us fix the notation and terminology: in this paper, $\N$ is the set of natural numbers (including $0$), empty sums are zero, by an \emph{algebra} we mean a unital associative \mbox{*-algebra} over the field of complex numbers, by an \emph{ideal} we mean a two-sided *-ideal, by \emph{homomorphism} we mean a unital *-homomorphism.

\section{The Theorem}\label{sec:1}

Let $n\in\N$ and $q\geq 0$. We define $\A(S^{2n+1}_q)$ as the universal algebra generated  by the elements $\{z_i,z_i^*\}_{i=0}^n$ with relations:
\begin{align}
z_jz_i &=qz_iz_j &&\forall\;0\leq i<j\leq n \;, \label{eq:qsphereA} \\
z_i^*z_j &=qz_jz_i^* &&\forall\;0\leq i\neq j\leq n \;, \label{eq:qsphereB} \\
z_i^*z_i &=z_iz_i^*+(1-q^2)\sum\nolimits_{j=i+1}^n z_jz_j^* 
    &&\forall\;0\leq i\leq n \;,\label{eq:qsphereC} \\
z_0z_0^*+z_1z_1^* &+\ldots+z_nz_n^*=1 \;. \label{eq:qsphereE}
\end{align}
Concretely, $\A(S^{2n+1}_q)$ is the quotient of the free algebra $\C\inner{\{z_i,z_i^*\}_{i=0}^n}$ by the ideal defined by the above relations. As the notation suggests, the \mbox{*-involution} is the map defined on generators by $z_i\mapsto z_i^*$ for all $0\leq i\leq n$.

For $q=0$, $\A(S^{2n+1}_q)$ is an example of Leavitt path algebra \cite{HS02,Dan23}.
When $q=1$, (\ref{eq:qsphereA}-\ref{eq:qsphereC}) tell us that the generators commute with each other, and we can think of $\A(S^{2n+1}_1)$ as the algebra of polynomial functions on the unit sphere in $\C^{n+1}$, the set of solutions of the equation \eqref{eq:qsphereE}. For $n=0$, $\A(S^1_q)$ is independent of $q$ and isomorphic to the algebra $\A(S^1)$ of polynomial functions on the unit circle. We can then focus on $n\geq 1$.

Even if the definition is valid for every non-negative number $q$, it is well-known that $\A(S^{2n+1}_q)$ is isomorphic to $\A(S^{2n+1}_{q^{-1}})$ for all $q>0$ (the reader can work out the explicit isomorphism as an exercise). For every $q\neq 1$ (and $n\geq 1$), the algebra $\A(S^{2n+1}_q)$ is non-commutative, hence not isomorphic to $\A(S^{2n+1}_1)$. This leaves us with the family of algebras with $n\geq 1$ and parameter $0\leq q<1$. The question is whether two algebras in this family can be isomorphic.

The rest of the paper is devoted to the proof of the following theorem.

\begin{thm}\label{thm:main}
Let $n\geq 1$, $q,q'\in\rinterval{0}{1}$ and $q\neq q'$. Then, the algebras $\A(S^{2n+1}_q)$ and $\A(S^{2n+1}_{q'})$ are not isomorphic.
\end{thm}

\section{The Proof}

From this point on, we will make the following assumptions:
\[
n\geq 1, \quad q,q'\in\rinterval{0}{1} \,,
\]
and whenever we write $q\neq 0$, we mean $0<q<1$.

\smallskip

We need some preliminary results on the algebra of the quantum $SU(2)$ group, which we denote by
 $\A(SU_q(2))$. This algebra is generated by four elements $\alpha,\beta,\alpha^*,\beta^*$ with relations
\begin{gather*}
\beta\alpha =q\alpha\beta ,\qquad
\beta^*\alpha =q\alpha\beta^* , \qquad
[\beta,\beta^*]=0 \\
\alpha\alpha^*+\beta\beta^*=\alpha^*\alpha+q^2\beta\beta^*=1 .
\end{gather*}
Note that these are the same relations of $\A(S^3_q)$, up to renaming the generators. However, it is convenient to introduce different notations to distinguish its generators from those of $\A(S^{2n+1}_q)$ for a general $n\geq 1$.

A surjective $*$-homomorphism $\A(S^{2n+1}_q)\to\A(SU_q(2))$ is defined on generators by
\[
z_0\mapsto\alpha , \qquad z_1\mapsto\beta , \qquad z_i\mapsto 0\;\forall\;i\geq 2 .
\]

\smallskip

For $j\in\Z$ and $k,l\in\N$, let us define
\begin{equation}\label{eq:ejkl}
e_{j,k,l}:=\begin{cases}
\alpha^j\beta^k(\beta^*)^l & \text{if }j\geq 0,\\
\beta^k(\beta^*)^l(\alpha^*)^{-j} & \text{if }j<0 .
\end{cases}
\end{equation}
If $q\neq 0$, the family $(e_{j,k,l})_{j\in\Z,k,l\in\N}$ is a basis of the complex vector space $\A(SU_q(2))$ (see e.g.~\cite[Prop.~4.4]{KS97}).

\begin{rem}
Note that the elements \eqref{eq:ejkl} are not linearly independent for $q=0$, since in this case it follows from the relation
$\alpha^*\alpha=1$ that $\alpha$ is an isometry, and $\beta\beta^*=1-\alpha\alpha^*$ is an idempotent. Thus, $\beta^k(\beta^*)^k=(\beta\beta^*)^k=\beta\beta^*$ for all $k\geq 1$.
\end{rem}

We now prove a series of results that involve the basis \eqref{eq:ejkl} and are valid if $q\neq 0$.
All the claims from here to Remark \ref{rem:4.5} are only valid for $q\neq 0$.

\begin{df}
For $m\in\N$, we define \ $V_m:=\mathrm{Span}\big\{ e_{j,k,l} : k+l\geq m \big\}$ .
\end{df}

We have a descending filtration
\[
\A(SU_q(2))=V_0\supseteq V_1\supseteq V_2\supseteq \ldots
\]
and clearly
\begin{equation}\label{eq:intersection}
\bigcap_{m\geq 0}V_m=0 .
\end{equation}

Despite being elementary to proof, it is useful to collect the following formulas in a Lemma for future reference.

\begin{lemma}\label{lem:aas}
For all $j,k\in\Z_+$, if $j\geq k$:
\[
\alpha^j(\alpha^*)^k=
\alpha^{j-k}\prod_{p=0}^{k-1}(1-q^{-2p}\beta\beta^*) ,
\]
and if $j<k$:
\[
\alpha^j(\alpha^*)^k=
(\alpha^*)^{k-j}\prod_{p=k-j}^{k-1}(1-q^{-2p}\beta\beta^*) .
\]
\end{lemma}

\begin{proof}
We start by writing
\[
\alpha^j(\alpha^*)^k=\alpha^{j-1}(1-\beta\beta^*)(\alpha^*)^{k-1}=\alpha^{j-1}(\alpha^*)^{k-1}(1-q^{-2(k-1)}\beta\beta^*)
\]
and then use induction on $k$ to prove the desired result.
\end{proof}

\begin{prop}\label{prop:3}
For all $m,m'\in\N$, one has
\begin{enumerate}
\item\label{lem3:1} $(V_m)^*=V_m$,
\item\label{lem3:2} $V_m\cdot V_{m'}\subseteq V_{m+m'}$.
\end{enumerate}
\end{prop}

\begin{proof}
\ref{lem3:1} From the defining relations of $SU_q(2)$ we see that, for all $j\in\Z$ and $k,l\in\N$, one has
\begin{equation}\label{eq:star}
(e_{j,k,l})^*=e_{-j,l,k} .
\end{equation}
Thus $(e_{j,k,l})^*\in V_m$ whenever $k+l\geq m$, i.e.~$(V_m)^*\subseteq V_m$, and then $V_m=((V_m)^*)^* \subseteq V_m^*$.

\medskip

\noindent
\ref{lem3:2}
It is enough to show that $e_{j,k,l}\cdot e_{j',k',l'}\in V_{m+m'}$ for all $j,j'\in\Z$ and $k,l,k',l'\in\N$ with $k+l\geq m$ and $k'+l'\geq m'$. If either $j$ or $j'$ is zero, or they are both non-zero and with the same sign, the thesis immediately follows from the equality
\[
e_{j,k,l}e_{j',k',l'}=q^{(\ldots)}e_{j+j',k+k',l+l'} ,
\]
where the power of $q$ in front is irrelevant.
Next, let $j>0$ and $j'<0$ (the case $j<0$ and $j'>0$ is similar and we omit it). Using Lemma \ref{lem:aas} we get
\begin{align*}
e_{j,k,l}e_{j',k',l'} &=q^{(\ldots)} \alpha^j(\alpha^*)^{-j'} e_{0,k+k',l+l'} \\
&=q^{(\ldots)} \underbrace{e_{j+j',k+k',l+l'}}_{\in V_{m+m'}}
\prod_{p=\max\{0,-j-j'\}}(1-q^{-2p}\beta\beta^*) ,
\end{align*}
where again we ignore the powers of $q$ in front coming from the commutation relations.
Since $V_{m+m'}\beta\beta^*\subseteq V_{m+m'}$, we get the thesis.
\end{proof}

\begin{rem}\label{rem:4.5}
It follows from Prop.~\ref{prop:3} that, for each $m\leq m'$, $V_{m'}$ is an ideal in $V_m$. And in particular they are all ideals in $\A(SU_q(2))$.
\end{rem}

If $A$ is an associative algebra, we denote by $[a,b]:=ab-ba$ the \emph{commutator} of two elements $a,b\in A$,
and call \emph{commutator ideal} the  ideal generated by all commutators, i.e.~the intersection of all ideals containing the set $[A,A]$.

\begin{lemma}\label{lem:6}
Let $A_1,A_2$ be algebras and denote by $J_1$ (resp.~$J_2$) the commutator ideal of $A_1$ (resp.~$A_2$).
If $\varphi:A_1\to A_2$ is a homomorphism, then $\varphi(J_1)\subseteq J_2$ and there is an induced
homomorphism
\[
[\varphi]:A_1/J_1\to A_2/J_2
\]
between quotient algebras defined by $[\varphi](a+J_1):=\varphi(a)+J_2$ for all $a\in A_1$.
Moreover, if $\varphi$ is surjective, then $[\varphi]$ is surjective as well.
\end{lemma}

\begin{proof}
For the first claim, it is enough to observe that $\varphi([a,b])=[\varphi(a),\varphi(b)]$ for all $a,b\in A_1$.
The inclusion $\varphi(J_1)\subseteq J_2$ implies that the induced map $[\varphi]$ is well-defined, and it is clearly a homomorphism. Assume that $\varphi$ is surjective and let $c\in A_2$. Then $c=\varphi(a)$ for some $a\in A_1$ and
$c+J_2=[\varphi](a+J_1)$, proving surjectivity of $[\varphi]$.
\end{proof}

\begin{lemma}\label{lemma:commid}
Let $0\leq q<1$.
The commutator ideal $I_n$ of $\A(S^{2n+1}_q)$ is generated by the elements $z_1,\ldots,z_n$. The quotient algebra $\A(S^{2n+1}_q)/I_n$ is generated by the unitary $[z_0]$ and is isomorphic to $\A(S^1)$.
\end{lemma}

\begin{proof}
Let $I_n'$ be the ideal generated by $z_1,\ldots,z_n$.
From \eqref{eq:qsphereB} it follows that, for all $i\neq j$, the element
\[
z_jz_i^*=(1-q)^{-1}[z_j,z_i^*]
\]
belongs to $I_n$, which means that $z_i^*z_j=z_jz_i^*+[z_i^*,z_j]\in I_n$ as well.
From \eqref{eq:qsphereC} we see that
\[
x_i:=\sum_{j\geq i}z_jz_j^*=(1-q^2)^{-1}[z_{i-1}^*,z_{i-1}]
\]
belongs to $I_n$ for all $1\leq i\leq n$. So, $z_iz_i^*=x_{i}-x_{i+1}\in I_n$ for all $1\leq i<n$, and $z_nz_n^*=x_n\in I_n$ as well.
From \eqref{eq:qsphereE}, for all $i \neq 0$,
\[
z_i=\sum_{j=0}^nz_jz_j^*z_i
\]
belongs to $I_n$, since $z_j^*z_i\in I_n$ for all $j\neq i$ and $z_iz_i^*\in I_n$. Thus,
\[
I_n'\subseteq I_n .
\]
We now study the quotient map $\pi:\A(S^{2n+1}_q)\to\A(S^{2n+1}_q)/I_n'$.
From the defining relations it follows that $\A(S^{2n+1}_q)/I_n'$ is commutative, hence
all the commutators are contained in the kernel of $\pi$, that is
\[
I_n\subseteq\ker(\pi)=I_n' .
\]
Thus, $I_n=I_n'$, and this proves the first claim of the lemma.

We now pass to the second claim. The quotient algebra is commutative and generated by $[z_0]$ and $[z_0]^*=[z_0^*]$, with relation
\[
[z_0][z_0]^*=[z_0z_0^*]=[z_0z_0^*+x_1]=[1] .
\]
Thus, there is a homomorphism $\phi:\A(S^1)\to\A(S^{2n+1}_q)/I_n$ that maps the unitary generator $u$ of the first algebra (the identity function on $S^1\subset\C$) to $[z_0]$. We must show that $\phi$ is an isomorphism.

For every $\lambda\in U(1)$ we have a character $\chi_\lambda$ of $\A(S^{2n+1}_q)$ defined by $\chi_\lambda(z_0)=\lambda$ and $\chi_\lambda(z_i)=0$ for $i>0$. This induces a homomorphism
$[\chi_\lambda]:\A(S^{2n+1}_q)/I_n\to\C$. Let $a=\sum_ka_ku^k\in \A(S^1)$, with $a_k\in\C$ and only finitely many coefficients different from zero. Assume that $\phi(a)=0$. Then, $a$ is in the kernel of $f_\lambda:=[\chi_\lambda]\circ\phi$. But
$f_\lambda(a)=\sum\nolimits_k a_k\lambda^k=0$ for all $\lambda\in U(1)$ implies that the polynomial
$P(t):=\sum\nolimits_k a_kt^k\in\C[t]$ has infinitely many zeroes, hence $P(t)=0$. This means $a_k=0$ for all $k$, so that $a=0$ and $\phi$ is injective.

The image of $\phi$ contains the generator $[z_0]$ of $\A(S^{2n+1}_q)/I_n$, hence $\phi$ is also surjective.
\end{proof}

\begin{cor}\label{cor}
Let $q\neq 0$. Then, the commutator ideal of $\A(SU_q(2))$ is $I_1=V_1$.
\end{cor}

\begin{proof}
Lemma \ref{lemma:commid} for $n=1$, after a change of notations, tells us that the commutator ideal of $\A(SU_q(2))$ 
is generated by $\beta$ and $\beta^*$. Since $\beta=e_{0,1,0}$ and $\beta^*=e_{0,0,1}$ belong to $V_1$, one has $I_1\subseteq V_1$. On the other hand, for all $j,k,l$ with $k+l\geq 1$, one has $e_{j,k,l}=q^{(\ldots)}e_{j,k-1,l}\beta$ if $k\geq 1$ or $e_{j,k,l}=q^{(\ldots)}e_{j,k,l-1}\beta^*$ if $j\geq 1$ (where the power of $q$ is irrelevant). In either case, $e_{j,k,l}\in I_1$, hence $V_1\subseteq I_1$.
\end{proof}

\begin{lemma}\label{lemma:everyunitary}
Let $u$ denote the unitary generator of $\A(S^1)$. Then, every unitary element of the algebra $\A(S^1)$ is of the form $\lambda u^j$ for some $\lambda\in U(1)$ and $j\in\Z$. In particular, every unitary generator of $\A(S^1)$ has either the form $\lambda u$ or $\lambda u^*$, for some $\lambda\in U(1)$.
\end{lemma}

\begin{proof}
Let $a=\sum_ka_ku^k$ be a non-zero element in $\A(S^1)$, with $a_k\in\C$ for all $k$ and only finitely many non-zero coefficients.
Call $i_0:=\min\{ i : a_i\neq 0 \}$ and 
$i_1:=\max\{i : a_i\neq 0 \}$. Then
\begin{align*}
aa^* &=\left(\sum\nolimits_{k=i_0}^{i_1}a_ku^k\right)\left(\sum\nolimits_{k=-i_1}^{-i_0}a_{-k}^*u^k\right) \\
&=a_{i_0}a_{i_1}^*u^{i_0-i_1}+\sum\nolimits_{k=i_0-i_1+1}^{i_1-i_0}(\ldots)u^k .
\end{align*}
If $aa^*=1$, since $a_{i_0}a_{i_1}^*\neq 0$ by construction, from the previous equality we deduce that $i_0=i_1$. Hence
$a=a_{i_0}u^{i_0}$. This proves the first claim. The second claim is obvious.
\end{proof}

\begin{lemma}\label{lem:10}
Let \mbox{$0\leq q<1$}, \mbox{$0<q'<1$} and let $\varphi:\A(S^{2n+1}_q)\to\A(SU_{q'}(2))$ be a surjective homomorphism.
Then, either
\begin{subequations}
\begin{equation}\label{eq:morA}
\varphi(z_0) =\lambda \alpha+x
\end{equation}
or
\begin{equation}\label{eq:morB}
\varphi(z_0) =\lambda \alpha^*+x
\end{equation}
\end{subequations}
for some $\lambda\in U(1)$ and $x\in V_1$.
\end{lemma}

\begin{proof}
Denote by $J_1$ the commutator ideal of $A_1:=\A(S^{2n+1}_q)$ and by $J_2$ the commutator ideal of $A_2:=\A(SU_{q'}(2))$. Recall that $A_1/J_1\cong\A(S^1)$ is generated by the unitary $[z_0]$ and
$A_2/J_2\cong\A(S^1)$ is generated by the unitary $u:=[\alpha]$ (cf.~Lemma \ref{lemma:commid}).
The induced map $[\varphi]:A_1/J_1\to A_2/J_2$ is surjective (cf.~Lemma \ref{lem:6}), which means that $[\varphi(z_0)]$ is a unitary generator of $A_2/J_2$.
From Lemma \ref{lemma:everyunitary} it follows that either $[\varphi(z_0)]=\lambda u=\lambda [\alpha]$ or 
$[\varphi(z_0)]=\lambda u^*=\lambda [\alpha^*]$, for some \mbox{$\lambda\in U(1)$}. Hence, either
$\varphi(z_0)=\lambda\alpha+x$ or 
$\varphi(z_0)=\lambda\alpha^*+x$ for some $\lambda\in U(1)$ and $x\in J_2$. But $J_2=V_1$ (cf.~Corollary \ref{cor}), hence the thesis.
\end{proof}

\begin{lemma}~\label{lem:finally}
Let \mbox{$0\leq q<1$} and \mbox{$0<q'<1$}.
If $q$ is not a positive power of $q'$, then any homomorphism $\A(S^{2n+1}_q)\to\A(SU_{q'}(2))$ vanishes on the commutator ideal $I_n$.
\end{lemma}

\begin{proof}
Let $\varphi:\A(S^{2n+1}_q)\to\A(SU_{q'}(2))$ be a homomorphism.
Fix $1\leq i\leq n$ and let $\varphi(z_i)=:y$. For $\varphi(z_0)$ we have the two cases in Lemma \ref{lem:10}.
Since $z_i$ is in the commutator ideal (Lemma \ref{lemma:commid}), $y$ must be in the commutator ideal (Lemma \ref{lem:6}), which means $y\in V_1$. Under the hypotheses of the Lemma, we shall prove by induction on $m\geq 1$ that
\begin{equation}\label{eq:indhyp}
y\in V_m \quad\Longrightarrow\quad y\in V_{m+1} .
\end{equation}
Thus, $y\in\bigcap_{m\geq 0}V_m$, and from \eqref{eq:intersection} it follows that $y=0$.
Since $\varphi(z_i)=0$ for all $1\leq i\leq n$ and $z_1,\ldots,z_n$ generate $I_n$, we get the desired result that $\varphi(I_n)=0$.

Now we prove the implication \eqref{eq:indhyp}. Assume that $y\in V_m$ (inductive hypothesis), i.e.~that
\begin{equation}\label{eq:becomes}
y=\sum_{k+l\geq m}y_{j,k,l}e_{j,k,l}
\end{equation}
for some $y_{j,k,l}\in\C$. We distinguish two cases.

\medskip

\noindent
{\bf (Case I)} Assume that  $\varphi$ satisfies \eqref{eq:morA}. Since $\varphi$ is a homomorphism,
\[
0=\varphi(z_iz_0-qz_0z_i)=
\varphi(z_n)\varphi(z_0)-q\varphi(z_0)\varphi(z_n)
=\lambda(y\alpha-q\alpha y)+\underbrace{yx-qxy}_{\in V_{m+1}} .
\]
Now, $x\in V_1$ and $y\in V_m$ imply that the underbraced element belongs to $V_{m+1}$ and, since the sum is zero, it must be also
\[
y\alpha-q\alpha y \in V_{m+1} .
\]
We compute the $q$-commutator (recall that $q$ can be $0$, but $q'\neq 0$):
\begin{align*}
y\alpha-q\alpha y ={} &\phantom{+}\sum_{\substack{j\geq 0\\[1pt] k+l\geq m}}
(q'^{k+l}-q)y_{j,k,l}e_{j+1,k,l} \\
&{+}\sum_{\substack{j<0\\[1pt] k+l\geq m}}y_{j,k,l}e_{0,k,l}
\Big((\alpha^*)^{-j}\alpha-qq'^{-k-l}\alpha(\alpha^*)^{-j}\Big) .
\end{align*}
By induction on $j$ one proves that $\alpha(\alpha^*)^{-j}$ and $(\alpha^*)^{-j}\alpha$ are both equal to $(\alpha^*)^{-j-1}$ modulo $V_2$, so
\[
y\alpha-q\alpha y =\!\!\sum_{\substack{j\geq 0\\[1pt] k+l\geq m}}
(q'^{k+l}-q)y_{j,k,l}e_{j+1,k,l}
+\!\!\sum_{\substack{j<0\\[1pt] k+l\geq m}}(1-qq'^{-k-l})y_{j,k,l}e_{j+1,k,l}
 \mod V_{m+2}
\]
This belongs to $V_{m+1}$ if and only if the components $(j,k,l)$ with $k+l=m$ vanish, that means
\[
(q'^m-q)y_{j,k,l}=0 \;\forall\;j\in\Z, k+l=m .
\]
Since $q$ is not a power of $q'$, we get $y_{j,k,l}=0$ for all $j\in\Z$ and $k+l=m$.  Thus, \eqref{eq:becomes} becomes
\[
y=\sum_{k+l\geq m+1}y_{j,k,l}e_{j,k,l} ,
\]
which proves that $y\in V_{m+1}$.

\medskip

\noindent
{\bf (Case II)} Assume $\varphi$ satisfies \eqref{eq:morB}. We start again from
\[
0=\varphi(z_iz_0-qz_0z_i)=
\varphi(z_n)\varphi(z_0)-q\varphi(z_0)\varphi(z_n)
=\lambda(y\alpha^*-q\alpha^* y)+\underbrace{yx-qxy}_{\in V_{m+1}} ,
\]
which now implies
\[
y\alpha^*-q\alpha^* y \in V_{m+1} .
\]
We compute the $q$-commutator:
\begin{align*}
y\alpha^*-q\alpha^* y ={} &\phantom{+}\sum_{\substack{j\leq 0\\[1pt] k+l\geq m}}
(1-qq'^{k+l})y_{j,k,l}e_{j-1,k,l} \\
&{+}\sum_{\substack{j>0\\[1pt] k+l\geq m}}y_{j,k,l}\Big(q'^{-k-l}\alpha^j\alpha^*-q\alpha^*\alpha^j\Big) 
e_{0,k,l} .
\end{align*}
By induction on $j$ one proves that $\alpha^j\alpha^*$ and $\alpha^*\alpha^j$ are both equal to $\alpha^{j-1}$ modulo $V_2$,
hence
\begin{multline*}
y\alpha^*-q\alpha^* y =\!\!\sum_{\substack{j\leq 0\\[1pt] k+l\geq m}}
(1-qq'^{k+l})y_{j,k,l}e_{j-1,k,l} \\
+\!\!\sum_{\substack{j>0\\[1pt] k+l\geq m}}(q'^{-k-l}-q) y_{j,k,l}e_{j-1,k,l}  \mod V_{m+2} .
\end{multline*}
This belongs to $V_{m+1}$ if and only if the components $(j,k,l)$ with $k+l=m$ vanish, that means
\[
(q'^{-m}-q)y_{j,k,l}=0 \;\forall\;j\in\Z, k+l=m .
\]
Since $q'^{-m}>1$ and $q<1$, one has $q'^{-m}-q\neq 0$. Hence, it must be
$y_{j,k,l}=0$  for all $j\in\Z$ and $k+l=m$, which means $y\in V_{m+1}$.
\end{proof}

We are now ready to prove the main theorem.

\begin{proof}[\protect{Proof of Theorem \ref{thm:main}}]
Since the statement is symmetric under the exchange of $q$ with $q'$, it is enough to consider the following two cases: \mbox{$0=q<q'<1$} and \mbox{$0<q'<q<1$}. In both cases, $q'$ is not zero, and $q$ is not a positive power of $q'$:
in the first case the statement is obvious, and in the second case it follows from  \mbox{$q'^m\leq q'<q$} for all $m\in\Z_+$.
By contradiction, assume that there exists an isomorphism 
$\A(S^{2n+1}_q)\to\A(S^{2n+1}_{q'})$. Composed with the quotient map
$\A(S^{2n+1}_{q'})\to\A(SU_{q'}(2))$ this gives a surjective homomorphism
\[
\varphi:\A(S^{2n+1}_q)\to\A(SU_{q'}(2)).
\]
It follows from Lemma \ref{lem:finally} that $\varphi$ vanishes on the commutator ideal $I_n$, hence it induces a surjective homomorphism
\[
\varphi:\A(S^{2n+1}_q)/I_n\to\A(SU_{q'}(2)).
\]
This is impossible, since $\A(S^{2n+1}_q)/I_n\cong\A(S^1)$ is commutative, while $\A(SU_{q'}(2))$ is not.
\end{proof}

\medskip

\begin{center}
\textsc{Acknowledgements}
\end{center}

\noindent
The Author is a member of INdAM-GNSAGA (Istituto Nazionale di Alta Matematica ``F.~Severi'') -- Unit\`a di Napoli and of INFN -- Sezione di Napoli.
%
%
%
%
%
This work is partially supported by the University of Naples Federico II under the grant FRA 2022 \emph{GALAQ: Geometric and ALgebraic Aspects of Quantization}.
%
%

\end{document}